\documentclass[11pt,reqno]{amsart}
\usepackage{fullpage}
\usepackage{a4wide}
\usepackage[utf8x]{inputenc}
\usepackage[T1]{fontenc}
\usepackage[english]{babel}
\usepackage{hyperref}
\usepackage{amsfonts,amsmath,amssymb,amsthm,amsrefs}
\usepackage{latexsym}
\usepackage{layout}
\usepackage{dsfont}
\usepackage{color}

\newtheorem{theorem}{Theorem}

\newtheorem{lemma}[theorem]{Lemma}
\newtheorem{corol}[theorem]{Corollary}

\theoremstyle{definition}

\theoremstyle{remark}
\newtheorem{remark}[theorem]{Remark}

\newcommand{\p}{\mathbb{P}}
\newcommand{\e}{\mathbb{E}}

\newcommand{\reals}{\mathbb{R}}
\newcommand{\ind}{\mathbf{1}}

\newcommand{\me}{\mathrm{e}}
\newcommand{\md}{\mathrm{d}}
\newcommand{\kappaq}{\kappa _{\lambda}}
\newcommand{\parisiantime}{\tilde{\kappa}}
\def\beq{\begin{eqnarray}} \def\eeq{\end{eqnarray}}

\def\al*#1{\begin{align*}#1\end{align*}}

\def\ga*#1{\begin{gather*}#1\end{gather*}}

\def\alat*#1#2{\begin{alignat*}{#1}#2\end{alignat*}}
\def\bea{\begin{eqnarray*}}
\def\eea{\end{eqnarray*}}
\def\ml*#1{\begin{multline*}#1\end{multline*}}

\allowdisplaybreaks

\begin{document}

\title[]{A note on Parisian ruin under a hybrid observation scheme}
\author[]{Mohamed Amine Lkabous}
\address{D\'epartement de math\'ematiques, Universit\'e du Qu\'ebec \`a Montr\'eal (UQAM), 201 av.\ Pr\'esident-Kennedy, Montr\'eal (Qu\'ebec) H2X 3Y7, Canada}
\email{lkabous.mohamed\_amine@courrier.uqam.ca}
\date{\today}
\begin{abstract}
In this paper, we study the concept of Parisian ruin under the hybrid observation scheme model introduced by Li et al. \cite{binetal2016}. Under this model, the process is observed at Poisson arrival times whenever the business is financially healthy and it is continuously observed when it goes below $0$. The Parisian ruin is then declared when the process stays below zero for a consecutive period of time greater than a fixed delay. We improve the result originally obtained in \cite{binetal2016} and we compute other fluctuation identities. All identities are given in terms of second-generation scale functions.
\end{abstract}

\maketitle

\maketitle

\section{Introduction}
Parisian ruin is a relaxation of the classical ruin since a delay is allowed before ruin occurs. More precisely, Parisian ruin occurs if the time spent below a pre-determined critical level for a consecutive period is greater than a pre-specified delay. Two types of Parisian ruin have been considered, one with fixed delays and another one with stochastic delays. Lkabous and Renaud \cite{lkabous_renaud2017} unified these two types of Parisian ruin into one called \textit{mixed Parisian ruin}. In this case, ruin is declared the first time an excursion into the red zone lasts longer than an implementation delay with a deterministic and a stochastic component. Recently, more definitions of Parisian ruin have been proposed. Cumulative Parisian ruin has been proposed in \cite{guerin_renaud_2015} and \cite{lkabous-renaud2018risks},; in that case, the \textit{race} is between a single deterministic clock and the sum of the excursions below the critical level. Moreover, in \cite{czarna_Renaud2015}, Parisian ruin with an ultimate bankruptcy level for Lévy insurance risk processes was considered. This type of ruin occurs if either the process goes below a predetermined negative level or if Parisian ruin with deterministic delays occurs.

Recently, Poisson observations problems have attracted considerable attention. In this case, the risk process is  monitored discretely at arrival epochs of an independent Poisson process, which can be interpreted as the observation times of the regulatory body, see Albrecher et al. \cite{albrecheretal2016} and Li et al. \cite{binetal2016} among others. A new definition of ruin has been studied in a spectrally negative Lévy setup by Li et al. \cite{binetal2016}. They introduced the idea of Parisian ruin under a hybrid observation scheme. More specifically, when the risk process is above $0$, it is monitored discretely at Poisson arrival times until a negative surplus is observed. Then, the process will be observed continuously and a grace period is granted for the insurance company to recover to a solvable level $a\geq 0$. Otherwise, Parisian ruin occurs.

In this paper, we study a Parisian ruin under a hybrid observation scheme for a general Lévy insurance risk model as defined in \cite{binetal2016}. We present a probabilistic analysis and simple resulting expressions for the two-sided exit problem, Laplace transform and the probability of Parisian ruin under a hybrid observation scheme in terms of the delayed scale functions introduced by Lkabous and Renaud \cite{lkabous_renaud2017}. Our approach is based on the expression of the Gerber–Shiu distribution at Parisian ruin with exponential implementation delays obtained in \cite{baurdoux_et_al_2015}, combined with some results in \cite{lkabousetal2016}.

The rest of the paper is organized as follows. In Section 2, we present the necessary background material on spectrally negative L\'evy processes and classical scale functions, including delayed scale functions and some fluctuation identities with delays already available in the literature. The main results are presented in Section 3, followed by a discussion on those results. In Section 4, we derive new technical identities and then provide proofs for the main results.  In the last section, we provide explicit computations of the probability of Parisian ruin under a hybrid observation scheme for Brownian risk model.
\subsection{Parisian ruin under a hybrid observation scheme}
 First, we denote the first passage times of $X$ above $b$
\begin{equation*}
\tau_{b}^{+}=\inf \left\{t \geq 0:\text{ }X_{t}\geq b \right\} .
\end{equation*}
For a standard L\'evy insurance risk process $X$, the time of Parisian ruin, with delay $r>0$, has been studied in \cite{loeffenetal2013}: it is defined as 
\begin{equation}\label{kappar}
\kappa_r = \inf \left\lbrace t > 0 \colon t - g_t > r \right\rbrace ,
\end{equation}
where $g_t = \sup \left\lbrace 0 \leq s \leq t \colon X_s \geq 0
\right\rbrace$. Then, Parisian ruin occurs the first time an excursion below zero lasts longer than the fixed implementation delay $r$. Loeffen et al. \cite{loeffenetal2013} obtained a very nice and compact expression for the probability of Parisian ruin : for $r>0$ and $x \in \reals$, we have
\begin{equation}\label{SNLPPr}
\p_{x} \left(\kappa_{r}<\infty \right) = 1 - \left(\e[X_{1}]\right))_{+} \frac{\int^{\infty}_{0}W(x+z)z\p \left( X_{r} \in \md z \right)}{\int^{\infty}_{0} z\p \left( X_{r} \in \md z \right)},
\end{equation}
where $(x)_{+}=\max(x,0)$ and $W$ is the so-called scale function of $X$ (see the definition in the next section).

Parisian ruin with stochastic delay has also been considered in
\cites{landriaultetal2011,landriaultetal2014,baurdoux_et_al_2015}. In this case, the fixed delay $r$ is replaced by an independent exponential random and it occurs the first time when the length an excursion below $0$ is longer than the exponential clock. 
It also corresponds to the  first passage time when $X$ is observed at Poisson arrival times below $0$, that is
\begin{align}\label{T}
T_{0}^- &= \min\{T_{i}>0 \colon X_{T_{i}}<0, \textit{ }i\in \mathbb{N} \}, 
\end{align}
where $T_{i}$ are the arrival times of an independent Poisson process of rate $\lambda>0$. We will use the notation $T_{0}^-$ instead of $\kappaq$. 

In \cite{binetal2016}, the time of Parisian ruin under a hybrid observation scheme with recover barrier $a\geq 0$ and a fixed delay $r>0$ is defined as
\begin{equation*}
\parisiantime_{a,r}^{\lambda }=\inf \left\{ t\in \left( T_{n},\tau_{a}^{+}\circ
\theta _{T_{n}}\right) :X_{T_{n}}<0\text{ and }t-T _{n}\geq r \text{, }n\in \mathbb{N}
\right\},
\end{equation*}
where $\theta$ is the Markov shift operator ($X_s\circ \theta_{t}=X_{s+t}$). In other words, when the risk process is above the level $a\geq 0$, it is monitored discretely at Poisson arrival times until a negative surplus is observed. Then, the process will be observed continuously and a fixed delay $r$ is granted for the insurance company to recover to the solvable level $a$. 

Li et al. \cite{binetal2016} obtained the following expression for the probability of Parisian ruin under a hybrid observation scheme using a standard probabilistic decomposition and using the technique of taking Laplace transform with respect to the delay $r$.
\begin{theorem}\label{Liformula}
For $r,\lambda>0$, $a\geq 0$ and $x \in \reals $, we have
\begin{equation*}
\mathbb{P}_{x}\left( \parisiantime_{a,r}^{\lambda }<\infty \right) =1-\psi ^{\prime }\left(
0+\right) \frac{\Phi _{\lambda }}{\lambda }Z\left(x,\Phi _{\lambda
}\right) -\psi ^{\prime }\left( 0+\right) \frac{\Phi _{\lambda }}{\lambda 
}\frac{Z\left( a,\Phi _{\lambda }\right) \lambda \int_{0}^{r}\mathrm{e}%
^{\lambda \left( r-s\right) }g_{x,a,\lambda }\left( s\right) \mathrm{d}s}{%
1-\lambda \int_{0}^{r}\mathrm{e}^{\lambda \left( r-s\right) }g_{a,a,\lambda
}\left( s\right) \mathrm{d}s},
\end{equation*}
where 
\begin{equation*}
g_{x,a,\lambda }\left( s\right) =\int_{a}^{\infty }\left( \frac{\Phi
_{\lambda }}{\lambda }Z\left( x,\Phi _{\lambda }\right) -W\left(
x+z-a\right) \right) \frac{z}{r}\mathbb{P}\left( X_{r}\in \mathrm{d}z\right).\end{equation*}
\end{theorem}

We want to improve on this result by making it more close to Equation \eqref{SNLPPr} using a probabilistic approach. Without loss of generality, we will assume that $a=0$ and we will write $\parisiantime_{0,r}^{\lambda }=\parisiantime_{r}^{\lambda }$.
\section{L\'{e}vy insurance risk processes}

We say that $X=\{X_t,t\geq 0\}$ is a L\'{e}vy insurance risk process if it
is a spectrally negative L\'evy process (SNLP) on the filtered probability
space $(\Omega,\mathcal{F},\{\mathcal{F}_t , t\geq0\}, \mathbb{P})$, that is
a process with stationary and independent increments and no positive jumps.
To avoid trivialities, we exclude the case where $X$ has monotone paths.

As the L\'{e}vy process $X$ has no positive jumps, its Laplace transform
exists: for all $\lambda, t \geq 0$, 
\begin{equation*}
\e \left[ \mathrm{e}^{\lambda X_t} \right] = \mathrm{e}^{t \psi(\lambda)} , 
\end{equation*}
where 
\begin{equation*}
\psi(\lambda) = \gamma \lambda + \frac{1}{2} \sigma^2 \lambda^2 +
\int^{\infty}_0 \left( \mathrm{e}^{-\lambda z} - 1 + \lambda z \ind%
_{(0,1]}(z) \right) \Pi(\mathrm{d}z) , 
\end{equation*}
for $\gamma \in \reals$ and $\sigma \geq 0$, and where $\Pi$ is a $\sigma$%
-finite measure on $(0,\infty)$ called the L\'{e}vy measure of $X$ such that 
\begin{equation*}
\int^{\infty}_0 (1 \wedge z^2) \Pi(\mathrm{d}z) < \infty . 
\end{equation*}
Finally, note that 
$\e \left[ X_1 \right] = \psi^{\prime }(0+)$. We will use the standard
Markovian notation: the law of $X$ when starting from $X_0 = x$ is denoted
by $\p_x$ and the corresponding expectation by $\e_x$. We write $\p$ and $\e$
when $x=0$.%

We recall Kendall's identity (see \cite[Corollary VII.3]{bertoin1996}): on $(0,\infty) \times (0,\infty)$, we have
\begin{equation}\label{eq:Kendall}
r \p(\tau_z^+ \in \mathrm{d}r) \mathrm{d}z = z \p(X_r \in \mathrm{d}z) \mathrm{d}r .
\end{equation}
Finally, we have the well-known expression for the Laplace transform of the first passage time above $b$
\begin{equation}\label{Laptrans}
\mathbb{E}_{x}\left[ \mathrm{e}^{-q \tau _{b}^{+}}\mathbf{1}_{\left\{ \tau _{b}^{+}<\infty \right\} }\right] =\me^{\Phi_q\left( x-b\right)}.
\end{equation}
\subsection{Scale functions}
We now present the definition of the scale functions $W_{q}$ and $Z_{q}$
of $X$. First, recall that there exists a function $\Phi \colon [0,\infty)
\to [0,\infty)$ defined by $\Phi_{q} = \sup \{ \lambda \geq 0 \mid
\psi(\lambda) = q\}$ (the right-inverse of $\psi$) such that 
\begin{equation*}
\psi ( \Phi_{q} ) = q, \quad q \geq 0 . 
\end{equation*}%
Now, for $q \geq 0$, the $q$-scale function of the process $X$ is defined as
the continuous function on $[0,\infty)$ with Laplace transform 
\begin{equation}  \label{def_scale}
\int_0^{\infty} \mathrm{e}^{- \lambda y} W_{q} (y) \mathrm{d}y = \frac{1}{%
\psi_q(\lambda)} , \quad \text{for $\lambda > \Phi_q$,}
\end{equation}
where $\psi_q(\lambda)=\psi(\lambda) - q$.
 This function is unique, positive and strictly increasing for $x\geq0$ and
is further continuous for $q\geq0$. We extend $W_{q}$ to the whole real
line by setting $W_{q}(x)=0$ for $x<0$. We write $W = W_{0}$ when $q=0$.\\
We also define another scale function $Z_{q}(x,\theta )$ by
\begin{equation}
Z_{q}(x,\theta )=\me^{\theta x}\left( 1-\psi_q (\theta )
\int_{0}^{x}\me^{-\theta y}W_{q}(y)\mathrm{d}y\right) ,\quad x\geq 0,
\end{equation}%
and $Z_{q}(x,\theta )=\me^{\theta x}$ for $x<0$. For $\theta=0$, we get 
\begin{equation}  \label{eq:zqscale}
Z_{q}(x,0)=Z_{q}(x) = 1 + q \int_0^x W_{q}(y)\mathrm{d }y, \quad x \in \mathbb{%
R}.
\end{equation}
Using \eqref{def_scale}, we can also re-write the scale function $Z_{q}(x,\theta )$ as follows 
\begin{equation}\label{Zv2}
Z_{q}(x,\theta )=\psi_q (\theta )\int_{0}^{\infty}\me^{-\theta y}W_{q}(x+y)\mathrm{d}y ,\quad x\geq 0,\textit{ }\theta\geq \Phi_q.
\end{equation}
It is known that, for any $x\in \reals$, 
\begin{equation}
\lim_{z\rightarrow \infty }\dfrac{W_{q}(x+z)}{W_{q}(z)}=\me^{\Phi_{q} x},
\label{CM1}
\end{equation}
and for $\theta>\Phi_q$,
\begin{equation}
\lim_{x\rightarrow \infty }\frac{Z_{q}(x,\theta)}{
W_{q}\left( x\right) }=\frac{\psi_q(\theta)}{\theta-\Phi_{q}}.  \label{limitZW}
\end{equation}
For the sake of compactness of the results, we will use the following functions. For $p,p+q\geq 0$ and $x \in \reals$, we set
\begin{align}\label{eq:convsnlp}
\mathcal{W}_{a}^{\left( p,q\right) }\left( x\right) &= W_{p}\left( x\right) +q\int_{a}^{x}W_{p+q}\left( x-y\right)W_{p}\left( y\right) \md y \nonumber\\ &= W_{p+q}\left( x\right)-q\int_{0}^{a}W_{p+q}\left( x-y\right) W_{p}\left( y\right) \md y,
\end{align}
where the second equation follows using the following identity obtained in \cite{loeffenetal2014}
$$
(s-p) \int_0^x W_p(x-y) W_s(y) \mathrm{d}y = W_s(x) - W_p(x).
$$
For later use, note that we can show 
\begin{equation}\label{LapscriptW}
\int_{0}^{\infty} \me^{-\theta z} \mathcal{W}_{a}^{\left( p,s\right)} \left(a+z\right) \mathrm{d}z = \frac{Z_{p}\left( a,\theta\right)}{ \psi_{p+s}(\theta)}, \quad \theta >\Phi_{p+s} .
\end{equation}
Finally, we have the following useful identities taken from \cite{albrecheretal2016} and \cite{MR3257360} and : for $\theta,p,q \geq 0$ 
\begin{equation}
\left( p-q\right) \int_{0}^{a}W_{p}\left( a-x\right) Z_{q}\left( x,\theta
\right) \md x=Z_{p}\left( a,\theta \right) -Z_{q}\left( a,\theta \right),
\label{EQ2}
\end{equation}
and for $s\geq 0$, $x>a$
\begin{equation}\label{convolution2}
(s-(p+q)) \int_a^x W_{s}(x-y)\mathcal{W}_{a}^{\left( p,q\right) }\left(y\right) \mathrm{d}y \\
=\mathcal{W}_{a}^{\left(p,s-p\right) }\left(x\right)-\mathcal{W}_{a}^{\left( p,q\right) }\left(x\right).
\end{equation}
\subsection{Delayed Scale functions}
Introduced by Lkabous and Renaud \cite{lkabous_renaud2017}, we use the \textit{$(r,s)$-delayed $p$-scale function of $X$} defined by
\begin{equation}\label{Exp1}
\Lambda^{\left( p\right)} \left( x;r,s\right) = \int_{0}^{\infty } \mathcal{W}_{z}^{\left( p+s,-s\right)} \left( x+z\right) \frac{z}{r} \mathbb{P} \left(X_{r} \in \mathrm{d}z \right),
\end{equation}
and using \eqref{eq:convsnlp}, we can rewrite $\Lambda^{\left( p\right)} \left( x;r,s\right)$ in the form
\begin{equation}\label{Exp2}
\Lambda^{\left( p\right)} \left( x;r,s\right)= \int_{0}^{\infty } \mathcal{W}_{x}^{\left( p,s\right)} \left( x+z\right) \frac{z}{r} \mathbb{P} \left(X_{r} \in \mathrm{d}z \right) .
\end{equation}
When $s=0$, we recover the function $\Lambda^{(p)}$ originally defined in \cite{loeffenetal2017}. To be more precise, when $s=0$, we have $\Lambda^{\left( p\right)} \left(x;r,0\right) = \Lambda^{\left( p\right)}\left( x,r\right)$, where
\begin{equation}\label{DS2}
\Lambda^{\left( p\right)}\left( x,r\right) = \int_{0}^{\infty} W_p\left( x+z\right) \frac{z}{r} \mathbb{P} \left( X_{r} \in \mathrm{d}z \right) ,
\end{equation}
and we write $\Lambda=\Lambda^{(0)}$ when $q=0$.\\
The delayed scale function \eqref{Exp1} appeared in \cite{lkabous_renaud2017} as the solution of two-sided exit problem for Parisian ruin with \textit{mixed delays}  : for $p \geq 0$, $b,r,\lambda>0$ and $x \leq b$ 
\begin{equation}\label{LaplaceTr3}
\mathbb{E}_{x} \left[ \mathrm{e}^{-p \tau_{b}^{+}} \mathbf{1}_{\left\{ \tau_{b}^{+} < \kappa_{r}^{\lambda} \right\}} \right] =\frac{\Lambda^{\left( p\right) }\left( x;r,\lambda\right)}{\Lambda^{\left( p\right) }\left( b;r,\lambda\right)},
\end{equation}
where $\kappa_{r}^{\lambda} $ is defined as
\begin{equation}  \label{kqr}
\kappa_{r}^{\lambda}=T^{-}_{0}\wedge \kappa_{r}.
\end{equation}
In the main results, the functions
 $\mathcal{S}^{\left( q \right) }$ and $\Theta ^{\left( q \right) }$ are given by 
\begin{eqnarray*}
\mathcal{S}^{\left( q \right) }\left( x,r\right)  &=&
Z_{q}\left( x\right)+\Lambda^{(q)} \left( x,r\right) -\Lambda^{(q)}\left( x;r,-q\right), \\
\Theta ^{\left( q \right) }\left( x;r,\lambda\right)  &=&\mathrm{e}^{(\lambda+q) r}Z_{q}\left( x,\Phi _{\lambda +q}\right) + \Lambda^{(q)}\left( x;r\right) -\Lambda^{\left( q \right) }\left( x;r,\lambda\right),
\end{eqnarray*}
and we note $\Theta^{\left( 0 \right)}=\Theta$ when $q=0$. 
\begin{remark}
It is interesting to note that the function $\Theta ^{\left(q \right) }\left( x;r,\lambda\right)$ is expressed in terms of scale functions $\Lambda^{\left( q\right)} \left( x;r,\lambda\right) $, $\Lambda^{\left(q\right)} \left( x;r\right)$ and $Z_{q}\left( x,\Phi _{\lambda +q}\right)$  related to Parisian ruins  \eqref{kqr}, \eqref{kappar} and \eqref{T} respectively. It could be called the \textit{hybrid} scale function. 
\end{remark}
\subsection{Parisian ruin exit problems.}

Here is a collection of known fluctuation identities for the spectrally negative L\'{e}vy processes in terms of their scale functions.
From \cite{baurdoux_et_al_2015} , we have the following expression of the Gerber–Shiu distribution at Parisian ruin with exponential implementation delays 
\begin{eqnarray}\label{G-S-Parisian}
\e_{x}\left[\me^{-qT_{0}^{-}},X_{T_{0}^{-}}\in \md y,T_{0}^{-}<\tau
_{a}^{+}\right] =\lambda \left( \frac{Z_{q}\left( x,\Phi _{\lambda+q}\right) }{Z_{q}\left( a,\Phi _{\lambda+q}\right) }\mathcal{W}_{a}^{\left(
q,\lambda \right) }\left( a-y\right) -\mathcal{W}_{x}^{\left(q,\lambda
\right) }\left( x-y\right) \right)\md y,
\end{eqnarray}
where $y\leq 0$ and $x\leq a$. When $a\rightarrow \infty$, we obtain 
\begin{multline}\label{G-S-Parisian2}
\e_{x}\left[\me^{-qT_{0}^{-}},X_{T_{0}^{-}}\in \md y,T_{0}^{-}<\infty\right] \\
= \left(\left(\Phi _{\lambda+q}-\Phi _{q}\right)Z_{q}\left( x,\Phi _{\lambda+q}\right) Z_{\lambda+q}\left(-y,\Phi _{q}\right)-\lambda\mathcal{W}_{x}^{\left(q,\lambda
\right) }\left( x-y\right) \right)\md y.
\end{multline}
From Landriault et al. \cite{landriaultetal2014}, we have for $x\leq a$, 
\begin{equation}\label{ParLap1}
\mathbb{E}_{x}\left[ \me^{-q\tau _{a}^{+}},\tau _{a}^{+}<T_{0}^{-}\right] =\frac{Z_{q}\left( x,\Phi _{\lambda +q}\right) }{Z_{q}\left( a,\Phi _{\lambda
+q}\right) }.
\end{equation}

\section{Main results}
We now present our main results. First, we derive the two-sided exit problem when a Parisian delay is added as an improvement over Theorem~\ref{Liformula}. We present a probabilistic analysis and simple resulting expressions for the two-sided exit problem, Laplace transform and the probability of the Parisian ruin under a hybrid observation scheme.

\begin{theorem}\label{ruinLap}
For $q\geq 0$, $b,r,\lambda>0$ and $x \leq b$, we have
\begin{eqnarray}\label{th1}
\mathbb{E}_{x}\left[ \mathrm{e}^{-q(\parisiantime _{r}^{\lambda }-r)}\ind_{\left\{\parisiantime_{r}^{\lambda }<\tau _{b}^{+}\right\} }\right]  &=&\frac{\lambda}{%
\lambda +q}\left( \mathcal{S}^{\left( q \right) }\left( x,r\right) -%
\frac{\Theta ^{\left( q \right) }\left( x;r,\lambda\right) }{\Theta
^{\left( q \right) }\left( b;r,\lambda\right) }\mathcal{S}^{\left( q\right) }\left( b,r\right) \right),
\end{eqnarray}

\begin{equation}\label{th2}
\mathbb{E}_{x}\left[ \mathrm{e}^{-q\tau _{b}^{+}}\ind_{\left\{ \tau_{b}^{+}<\parisiantime_{r}^{\lambda }\right\} }\right] =\frac{ \Theta ^{\left( q
\right) }\left( x;r,\lambda\right) }{\Theta ^{\left( q \right) }\left( b;r,\lambda\right) },
\end{equation}
and
\begin{equation}\label{th3}
\e_{x}\left[ \mathrm{e}^{-q(\parisiantime_{r}^{\lambda }-r)}\ind_{\left\{\parisiantime_{r}^{\lambda }<\infty \right\} }\right] =\frac{\lambda }{\lambda +q}\left( \mathcal{S}^{\left( q\right) }\left( x,r\right) -\Theta ^{\left( q \right) }\left( x;r,\lambda\right)\times\tilde{S}^{\left( q,\lambda\right)}\left( r\right) \right),
\end{equation}
where
\begin{equation*}
\mathcal{\tilde{S}}^{\left( q,\lambda\right) }\left( r\right) =\frac{q/\Phi _{q}-\int_{0}^{\infty }\left(
Z\left( z,\Phi _{q}\right) -\mathrm{e}^{\Phi _{q}z}\right) \frac{z}{r}%
\mathbb{P}\left( X_{r}\in \mathrm{d}z\right) }{\mathrm{e}^{\left( \lambda+q\right) r}\lambda /\left( \Phi _{\lambda +q}-\Phi _{q}\right)-\int_{0}^{\infty }\left(Z_{q+\lambda }\left(z,\Phi _{q}\right) - \mathrm{e}^{\Phi _{q}z}\right) \frac{z}{r}\mathbb{P}\left( X_{r}\in \mathrm{d}%
z\right) }.
\end{equation*}
\end{theorem}
Setting $q=0$ in \eqref{th3}, we obtain the following new expression for the probability of Parisian ruin.
\begin{corol}\label{ruinP}
For $x \in \reals$ and $\lambda,r>0$, we have
\begin{equation}\label{ruinP1}
\p_{x}\left( \parisiantime^{\lambda}_{r}<\infty \right) =1-\Theta\left( x,r;\lambda\right)\times\tilde{S}^{\left(0,\lambda\right)}\left( r\right).
\end{equation}
\end{corol}

\begin{remark}
When $ \psi^{\prime }(0+)>0$, we have 
\begin{equation*}
\mathcal{\tilde{S}}^{\left( 0,\lambda \right) }\left( r\right) =\frac{%
\mathbb{E}\left[ X_{1}\right] }{\mathrm{e}^{\lambda r}\lambda /\Phi
_{\lambda }-\int_{0}^{\infty }\left( Z_{\lambda }\left( z\right) -1\right) 
\frac{z}{r}\mathbb{P}\left( X_{r}\in \mathrm{d}z\right) }.
\end{equation*}
Then,
\begin{equation}\label{ruinP1v2}
\p_{x}\left( \parisiantime^{\lambda}_{r}<\infty \right) =1-\mathbb{E}\left[ X_{1}\right]\frac{\Theta\left( x;r,\lambda\right)}{\mathrm{e}^{\lambda r}\lambda /\Phi
_{\lambda }-\int_{0}^{\infty }\left( Z_{\lambda }\left( z\right) -1\right) 
\frac{z}{r}\mathbb{P}\left( X_{r}\in \mathrm{d}z\right) }.
\end{equation}
Our expression of the probability of Parisian ruin has a similar structure as Equation \eqref{SNLPPr} and it is different than the one in Theorem~\ref{Liformula}, which is because in the proof in \cite{binetal2016} they use a different approach.
\end{remark}
\subsection{Discussion on the results}\label{sect:discussion}
The fluctuation identities in Theorem~\ref{ruinLap}  have a similar structure as the Parisian fluctuation identities. Indeed, in Equation~\eqref{th1}, the functions $S ^{\left( q\right) }(\cdot,r)$ and $\Theta^{\left( q\right) }(\cdot;r,\lambda)$ play a similar role as the one played by the scale functions $Z_q(\cdot )$ and $Z_q(\cdot,\Phi_{\lambda+q} )$ respectively in the following Parisian fluctuation identity: for $x\leq b$, we have
\begin{equation}\label{iden1}
\mathbb{E}_{x}\left[ \me^{-qT_{0}^{-}}\ind_{\left\{T_{0}^{-}<\tau _{b}^{+}\right\}}\right] =\frac{%
\lambda }{\lambda +q }\left( Z_{q }\left( x\right) -\frac{
Z_{q}\left( x,\Phi _{\lambda +q}\right) }{Z_{q}\left( b,\Phi
_{\lambda+q}\right) }Z_{q }\left(b\right) \right).
\end{equation}

First, we will show that the function  $S ^{\left( q\right) }(\cdot,r)$ converges, as $r\rightarrow 0$, to the scale function $Z_q(\cdot)$.
Taking Laplace transforms in $r$, together with Kendall's identity, Tonelli's theorem, \eqref{Laptrans} and \eqref{LapscriptW}, we have 
\begin{equation*}
\int_{0}^{\infty }\mathrm{e}^{-\theta r}\left( \mathrm{e}^{-qr}\Lambda
^{\left( q\right) }\left( x;r\right) \right) \mathrm{d}r=\int_{0}^{\infty }%
\mathrm{e}^{-\Phi_{\theta +q }z}W_{q}\left( x+z\right) \mathrm{d}z=\dfrac{Z_{q}\left( x,\Phi _{\theta +q}\right) }{\theta},
\end{equation*}%
and 
\begin{eqnarray*}
\int_{0}^{\infty }%
\mathrm{e}^{-\theta r}\left( \mathrm{e}^{-qr}\Lambda ^{\left( q\right)
}\left( x;r,-q\right) \right) \mathrm{d}r=\frac{Z_{q}\left( x,\Phi_{\theta +q}\right)  }{\theta+q }.
\end{eqnarray*}
Then, using the Initial Value Theorem, we obtain 
\begin{multline*}
\lim_{r\rightarrow 0}\mathrm{e}^{-qr}\left( \Lambda ^{\left( q\right)
}\left( x;r\right) -\Lambda ^{\left( q\right) }\left( x;r,-q\right) \right) 
\\
=\lim_{\theta \rightarrow \infty }\theta \int_{0}^{\infty }\mathrm{e}%
^{-(\theta +q)r}\left( \Lambda ^{\left( q\right) }\left( x;r\right) -\Lambda
^{\left( q\right) }\left( x;r,-q\right) \right) \mathrm{d}r \\
=\lim_{\theta \rightarrow \infty }\frac{q}{\theta +q}Z_{q}\left( x,\Phi_{\theta +q}\right) =0.
\end{multline*}
where the last equality follows form the fact that 
\begin{equation*}
\lim_{\theta \rightarrow \infty }\frac{q}{\theta +q}Z_{q}\left( x,\Phi_{\theta +q}\right) =\lim_{\theta \rightarrow \infty }\left( \frac{q\theta }{%
\left( \theta +q\right) \Phi_{\theta +q}}\right) \left( \frac{\Phi_{\theta +q}}{\theta }Z_{q}\left( x,\Phi_{\theta +q}\right) \right) ,
\end{equation*}
and
\begin{equation*}
\lim_{\theta \rightarrow \infty } \frac{\Phi_{\theta +q}}{%
\theta }Z_{q}\left( x,\Phi_{\theta +q}\right)=\lim_{\theta \rightarrow \infty } \Phi_{\theta +q}\int_{0}^{\infty}\me^{-\Phi_{\theta +q} y}W_{q}(x+y)\mathrm{d}y =W_{q}\left( x\right),
\end{equation*}
where the last equality follows from the Initial Value Theorem.
Then,
\begin{eqnarray*}
\lim_{r\rightarrow 0}\mathrm{e}^{-qr}\mathcal{S}^{\left( q,\lambda \right)
}\left( x,r\right) &=&Z_{q}\left( x\right) .
\end{eqnarray*}
By the same steps, we can also show that
\begin{equation*}
\lim_{r\rightarrow 0}\Theta ^{\left(q \right) }\left( x;r,\lambda\right)
=Z_{q}\left( x,\Phi _{\lambda +q}\right).
\end{equation*}
Thus, 
\begin{eqnarray*}
\lim_{r\rightarrow 0}\mathbb{E}_{x}\left[ \mathrm{e}^{-q\parisiantime _{r}^{\lambda }}\ind_{\left\{\parisiantime_{r}^{\lambda }<\tau _{b}^{+}\right\} }\right] =\mathbb{E}_{x}\left[ \me^{-qT_{0}^{-}}\ind_{\left\{T_{0}^{-}<\tau _{b}^{+}\right\}}\right].
\end{eqnarray*}
\section{Proofs }\label{S:proofs}
The proofs of our main results are based on technical but important lemmas, as well as more standard probabilistic decompositions. We use the expression of the Gerber–Shiu distribution at Parisian ruin with exponential implementation delays in \cite{baurdoux_et_al_2015} and some results in \cite{lkabousetal2016} to obtain our key lemma (Lemma~\ref{L:L5} below). First, our main interest is deriving a closed formula for the following identity
$$\mathbb{E}_{x}\left[ \mathrm{e}^{-qT_{0}^{-}}W_{p}\left(
X_{T_{0}^{-}}+z\right) \mathbf{1}_{\left\{ T_{0}^{-}<\tau _{b}^{+}\right\} }%
\right],$$
where $p,q\geq 0$, $x\leq b$ and $W_p$ is the scale function of $X$.
\begin{lemma}\label{L:L5}
For $p,q,b\geq 0$, $\lambda, r,z>0$ and $x\leq b$, we have
\begin{multline}
\mathbb{E}_{x}\left[ \mathrm{e}^{-qT_{0}^{-}}W_{p}\left(
X_{T_{0}^{-}}+z\right) \mathbf{1}_{\left\{ T_{0}^{-}<\tau _{b}^{+}\right\} }%
\right]\\ =\dfrac{\lambda }{p-\left( q+\lambda \right)}\frac{Z_{q}\left(
x,\Phi _{\lambda +q}\right) }{Z_{q}\left( b,\Phi _{\lambda +q}\right) }%
\left( \mathcal{W}_{b}^{\left( q,p-q\right) }\left( b+z\right) -\mathcal{W}%
_{b}^{\left( q,\lambda \right) }\left( b+z\right) \right)   \label{LL1} \\
-\dfrac{\lambda }{p-\left( q+\lambda \right) }\left( \mathcal{W}%
_{x}^{\left( q,p-q\right) }\left( x+z\right) -\mathcal{W}_{x}^{\left(
q,\lambda \right) }\left( x+z\right) \right) ,
\end{multline}
and
\begin{multline}
\mathbb{E}_{x}\left[ \mathrm{e}^{-qT_{0}^{-}}W_{p}\left(
X_{T_{0}^{-}}+z\right) \mathbf{1}_{\left\{ T_{0}^{-}<\infty \right\} }\right]\\
=\dfrac{\left( \Phi _{\lambda +q}-\Phi _{q}\right) }{p-\left( q+\lambda \right) }Z_{p}\left( x,\Phi _{\lambda +p}\right) \left(Z_{p}\left( x+z,\Phi _{q}\right)-Z_{\lambda+q}\left(x+z,\Phi _{q}\right)\right)
\label{LL2} \\
-\dfrac{\lambda }{p-\left( q+\lambda \right)}\left( \mathcal{W}%
_{x}^{\left( q,p-q\right) }\left( x+z\right) -\mathcal{W}_{x}^{\left(
q,\lambda \right) }\left( x+z\right) \right).
\end{multline}
\end{lemma}
%
\begin{proof}
Using \eqref{G-S-Parisian}, we have
\begin{multline}
\mathbb{E}_{x}\left[ \mathrm{e}^{-qT_{0}^{-}}W_{p}\left(
X_{T_{0}^{-}}+z\right) \mathbf{1}_{\left\{ T_{0}^{-}<\tau _{b}^{+}\right\} }%
\right]  \\
=\int_{-\infty }^{0}W_{p}\left( y+z\right) \mathbb{E}_{x}\left[ \mathrm{e}%
^{-qT_{0}^{-}},X_{T_{0}^{-}}\in \mathrm{d}y,T_{0}^{-}<\tau _{b}^{+}\ \right]  \\
=\lambda \frac{Z_{q}\left( x,\Phi _{\lambda +q}\right) }{Z_{q}\left( b,\Phi
_{\lambda +q}\right) }\int_{-\infty }^{0}W_{p}\left( y+z\right) \mathcal{W}%
_{b}^{\left( q,\lambda \right) }\left( b-y\right) \mathrm{d}y \\
-\lambda \int_{-\infty }^{0}W_{p}\left( y+z\right) \mathcal{W}_{x}^{\left(
q,\lambda \right) }\left( x-y\right) \mathrm{d}y \\
=\frac{Z_{q}\left( x,\Phi _{\lambda +q}\right) }{Z_{q}\left( b,\Phi
_{\lambda +q}\right) }\lambda \int_{0}^{z}W_{p}\left( z-y\right) \mathcal{W}%
_{b}^{\left( q,\lambda \right) }\left( b+y\right) \mathrm{d}y \\
-\lambda \int_{0}^{z}W_{p}\left( z-y\right) \mathcal{W}_{x}^{\left(
q,\lambda \right) }\left( x+y\right) \mathrm{d}y \\
\text{ \ \ \ \ \ \ \ \ \ \ \ \ \ \ \ \ \ \ \ \ \ \ \ }=\dfrac{\lambda }{
p-\left( q+\lambda \right)}\frac{Z_{q}\left( x,\Phi _{\lambda +q}\right) 
}{Z_{q}\left( b,\Phi _{\lambda +q}\right) }\left( \mathcal{W}_{b}^{\left(
q,p-q\right) }\left( b+z\right) -\mathcal{W}_{b}^{\left( q,\lambda \right)
}\left( b+z\right) \right)  \\
-\dfrac{\lambda }{p-\left( q+\lambda \right)}\left( \mathcal{W}%
_{x}^{\left( q,p-q\right) }\left( x+z\right) -\mathcal{W}_{x}^{\left(
q,\lambda \right) }\left( x+z\right) \right) ,
\end{multline}
where in the last equality we used identity \eqref{convolution2}. The second identity follows using \eqref{G-S-Parisian2} and \eqref{EQ2}.
\end{proof}
\begin{remark}
The last lemma generalizes the classical identity that can be found in numerous references (see \cite{landriaultetal2018} for more details on the convergence when $\lambda\rightarrow \infty$) , see e.g.  \cite{loeffenetal2014}, that is : for any $p,q\geq 0$ and $a\leq x \leq b$, we have 
\begin{equation} \label{classiden}
\e_{x}\left[ \me^{-q\tau _{a}^{-}}W_{p}\left( X_{\tau_{a}^{-}}+z\right) \ind_{\left\{ \tau _{a}^{-}<\tau _{b}^{+}\right\} }\right] =
\mathcal{W}_{a+z}^{\left(p,q-p\right) }\left(x+z\right) -\frac{W_q\left( x-a\right) }{W_q\left( b-a\right) }\mathcal{W}_{a+z}^{\left(p,q-p\right) }\left( b+z\right).
\end{equation}

\end{remark}
Before presenting our main lemma below, we have the following identities taken from \cite{lkabousetal2016}. 
\begin{equation}\label{L22}
\e_x\left[\me^{-q\tau^+_0}\ind_{\left\lbrace \tau^+_0<r \right\rbrace }\right]=\me^{-qr}\Lambda^{(q)}(x,r), \textit{  }x\leq0,
\end{equation}
and
\begin{equation} \label{L11}
\Lambda ^{\left(q\right)}\left(0,r\right)= \me^{qr}.
\end{equation}
The following identities are new and crucial for the proof of our main results.
\begin{lemma}\label{mainlem}
For $p,q,b\geq 0$, $\lambda, r>0$ and $x \leq b$, we have 
\begin{multline} \label{Lem1} 
\mathbb{E}_{x}\left[\mathrm{e}^{-qT_{0}^{-}}\Lambda(X_{T_{0}^{-}},r) \ind_{\left\{ T_{0}^{-}<\tau _{b}^{+}\right\}}\right]\\  =\frac{\lambda }{\lambda +q}\left(\Lambda^{(q)}\left(x;r,-q\right) -\Lambda^{(q)}\left( x;r,\lambda\right) -\frac{Z_{q}\left( x,\Phi _{\lambda
+q}\right) }{Z_{q}\left( b,\Phi _{\lambda +q}\right) } \left( \Lambda^{(q)}\left( b;r,-q\right) -\Lambda^{(q)}\left( b;r,\lambda\right)\right)\right),
\end{multline}%
and
\begin{multline}\label{Lem2} 
\e_{x}\left[ \mathrm{e}^{-qT_{0}^{-}}\Lambda^{(q)}(X_{T_{0}^{-}},r)
\ind_{\left\{ T_{0}^{-}<\tau _{b}^{+}\right\} }\right] \\=\Lambda^{(q)}\left( x;r\right) -\Lambda^{\left( q \right) }\left( x;r,\lambda\right) -\frac{Z_{q}\left( x,\Phi _{\lambda
+q}\right) }{Z_{q}\left( b,\Phi _{\lambda +q}\right) } \left(\Lambda^{(q)}\left( b;r\right) -\Lambda^{\left( q \right) }\left( b;r,\lambda\right)\right).
\end{multline}
\end{lemma}
\begin{proof}
Using \eqref{LL1} for $q=0$ and Tonelli’s theorem, we obtain
\begin{multline*}
\mathbb{E}_{x}\left[\mathrm{e}^{-qT_{0}^{-}}\Lambda(X_{T_{0}^{-}},r) \ind_{\left\{ T_{0}^{-}<\tau _{b}^{+}\right\}}\right]\\  =\mathbb{E}_{x}\left[ \mathrm{e}^{-qT_{0}^{-}}\int_{0}^{\infty
}W\left( X_{T_{0}^{-}}+z\right) \frac{z}{r}\mathbb{P}\left( X_{r}\in \mathrm{%
d}z\right)\ind_{\left\{ T_{0}^{-}<\tau _{b}^{+} \right\} }\right]  \\
=\frac{\lambda }{\lambda +q}\int_{0}^{\infty }\left( \mathcal{W}_{x}^{\left( q,-q\right) }\left(
x+z\right) -\mathcal{W}_{x}^{\left( q,\lambda \right) }\left( x+z\right)
\right) \frac{z}{r}\mathbb{P}\left( X_{r}\in \mathrm{d}z\right)  \\
-\frac{\lambda }{\lambda +q}\frac{Z_{q}\left( x,\Phi _{\lambda +q}\right) }{Z_{q}\left( b,\Phi
_{\lambda +q}\right) }\int_{0}^{\infty }\left( \mathcal{W}_{b}^{\left(
q,-q\right) }\left( b+z\right) -\mathcal{W}_{b}^{\left(q,\lambda \right)
}\left( b+z\right) \right) \frac{z}{r}\mathbb{P}\left( X_{r}\in \mathrm{d}%
z\right). 
\end{multline*}
Also, using \eqref{LL2} with $p=q$ and Tonelli’s theorem, we have
\begin{multline*}
\mathbb{E}_{x}\left[\mathrm{e}^{-qT_{0}^{-}}\Lambda^{(q)}(X_{T_{0}^{-}},r) \ind_{\left\{ T_{0}^{-}<\tau _{b}^{+}\right\}}\right]\\=\mathbb{E}_{x}\left[ \mathrm{e}^{-qT_{0}^{-}}\int_{0}^{\infty
}W_{q}\left( X_{T_{0}^{-}}+z\right) \frac{z}{r}\mathbb{P}\left( X_{r}\in 
\mathrm{d}z\right)\ind_{\left\{ T_{0}^{-}<\tau_{b}^{+} \right\} }\right]  \\
=\int_{0}^{\infty }\left( \mathcal{W}_{x}^{\left(
q,0\right) }\left( x+z\right) -\mathcal{W}_{x}^{\left( q,\lambda \right)
}\left( x+z\right) \right) \frac{z}{r}\mathbb{P}\left( X_{r}\in \mathrm{d}%
z\right)  \\
-\frac{Z_{q}\left( x,\Phi _{\lambda +q}\right) }{%
Z_{q}\left( b,\Phi _{\lambda +q}\right) }\int_{0}^{\infty }\left( \mathcal{W}%
_{b}^{\left( q,0\right) }\left( b+z\right) -\mathcal{W}_{b}^{\left(
q,\lambda \right) }\left( b+z\right) \right) \frac{z}{r}\mathbb{P}\left(
X_{r}\in \mathrm{d}z\right),
\end{multline*}
and the result follows.
\end{proof} 
\subsection{Proof of Theorem~\ref{ruinLap}}
 The steps of the proof of Theorem~\ref{ruinLap} is based on the new Lemma~\ref{mainlem} together with standard probabilistic decompositions.
For $x<0$, from the strong Markov property and the fact that X is skip-free upward, we have
\begin{eqnarray}\label{xnegLap}
\mathbb{E}_{x}\left[ \mathrm{e}^{-q\parisiantime_{r}^{\lambda }}\ind_{\left\{\parisiantime_{r}^{\lambda }<\tau _{b}^{+}\right\} }\right] =\me^{-qr}\p_x\left(\tau^{+}_{0}>r \right)+\mathbb{E}\left[ \mathrm{e}
^{-q\parisiantime _{r}^{\lambda }}\ind_{\left\{\parisiantime_{r}^{\lambda }<\tau
_{b}^{+}\right\} }\right] \mathbb{E}_{x}\left[ \mathrm{e}^{-q\tau
_{0}^{+}}\ind_{\left\{ \tau _{0}^{+}<r\right\} }\right]. 
\end{eqnarray}
Consequently, for $0\leq x\leq b$, we get
\begin{eqnarray*}
\mathbb{E}_{x}\left[ \mathrm{e}^{-q\parisiantime_{r}^{\lambda }}\ind_{\left\{\parisiantime_{r}^{\lambda }<\tau _{b}^{+}\right\} }\right]  &=&\e_{x}\left[ \mathrm{e}^{-qT_{0}^{-}}\e_{X_{T^{-}_{0}}}%
\left[ \mathrm{e}^{-q\parisiantime_{r}^{\lambda }}\ind_{\left\{\parisiantime_{r}^{\lambda}<\tau _{b}^{+}\right\} }\right]\ind_{\left\{ T_{0}^{-}<\tau _{b}^{+}\right\} }\right].
\end{eqnarray*}
Injecting~\eqref{xnegLap} in the last expectation, we have, for all $x\in\reals$
\begin{eqnarray}\label{allx}
\mathbb{E}_{x}\left[ \mathrm{e}^{-q\parisiantime_{r}^{\lambda }}\ind_{\left\{\parisiantime_{r}^{\lambda }<\tau _{b}^{+}\right\} }\right]  &=&\mathrm{e}^{-qr}\mathbb{E}_{x}%
\left[ \mathrm{e}^{-qT_{0}^{-}}\mathbb{P}_{X_{T_{0}^{-}}}\left( \tau
_{0}^{+}>r\right)\ind_{\left\{ T_{0}^{-}<\tau _{b}^{+}\right\} }\right] \nonumber \\
&&+\mathbb{E}\left[ \mathrm{e}^{-q\parisiantime_{r}^{\lambda }}\ind_{\left\{\parisiantime_{r}^{\lambda }<\tau _{b}^{+}\right\} }\right] \mathbb{E}_{x}\left[ \mathrm{e}%
^{-qT_{0}^{-}}\mathbb{E}_{X_{T_{0}^{-}}}\left[ \mathrm{e}^{-q\tau _{0}^{+}}\ind_{%
\left\{ \tau _{0}^{+}<r\right\} }\right]\ind_{\left\{ T_{0}^{-}<\tau
_{b}^{+}\right\} }\right] \nonumber \\
&=&\mathrm{e}^{-qr}\mathbb{E}_{x}\left[ \mathrm{e}^{-qT_{0}^{-}}\ind_{\left\{
T_{0}^{-}<\tau _{b}^{+}\right\} }\right]\nonumber \\
&&-\me^{-qr}\mathbb{E}_{x}\left[\mathrm{e}^{-qT_{0}^{-}}\Lambda(X_{T_{0}^{-}},r) \ind_{\left\{ T_{0}^{-}<\tau _{b}^{+}\right\}}\right]\nonumber \\
&&+\me^{-qr}\mathbb{E}\left[ \mathrm{e}^{-q\parisiantime_{r}^{\lambda }}\ind_{\left\{\parisiantime_{r}^{\lambda }<\tau _{b}^{+}\right\} }\right] \mathbb{E}_{x}\left[\mathrm{e}^{-qT_{0}^{-}}\Lambda^{(q)}(X_{T_{0}^{-}},r) \ind_{\left\{ T_{0}^{-}<\tau _{b}^{+}\right\}}\right],
\end{eqnarray}
where in the last equality we used identity \eqref{L22}. For $x=0$ and using the last equation
\begin{equation}\label{lapx=0}
\mathbb{E}\left[ \mathrm{e}^{-q\parisiantime_{r}^{\lambda }}\ind_{\left\{\parisiantime_{r}^{\lambda }<\tau _{b}^{+}\right\} }\right] =\frac{\mathrm{e}^{-qr}\mathbb{E}%
\left[ \mathrm{e}^{-qT_{0}^{-}}\ind_{\left\{ T_{0}^{-}<\tau _{b}^{+}\right\} }%
\right] -\me^{-qr}\mathbb{E}\left[\mathrm{e}^{-qT_{0}^{-}}\Lambda(X_{T_{0}^{-}},r) \ind_{\left\{ T_{0}^{-}<\tau _{b}^{+}\right\}}\right]}{1-\me^{-qr}\mathbb{E}\left[\mathrm{e}^{-qT_{0}^{-}}\Lambda^{(q)}(X_{T_{0}^{-}},r) \ind_{\left\{ T_{0}^{-}<\tau _{b}^{+}\right\}}\right] }.
\end{equation}
From \eqref{L11}, \eqref{Lem1} and \eqref{Lem2}, we have
\begin{equation}\label{num}
\mathbb{E}\left[\mathrm{e}^{-qT_{0}^{-}}\Lambda(X_{T_{0}^{-}},r) \ind_{\left\{ T_{0}^{-}<\tau _{b}^{+}\right\}}\right]=\frac{\lambda }{\lambda +q}\left(1-\mathrm{e}^{\left(\lambda +q\right) r} -\frac{\Lambda^{(q)}\left(b;r,-q\right)-\Lambda^{(q)}\left(b;r,\lambda\right)}{Z_{q}\left( b,\Phi _{\lambda +q}\right) }\right) ,
\end{equation}
and
\begin{equation}\label{den}
\mathbb{E}\left[\mathrm{e}^{-qT_{0}^{-}}\Lambda^{(q)}(X_{T_{0}^{-}},r) \ind_{\left\{ T_{0}^{-}<\tau _{b}^{+}\right\}}\right] = \me^{qr}-\mathrm{e}^{ (\lambda+q) r} -\frac{\Lambda^{(q)}\left(b;r\right)-\Lambda^{(q)}\left(b;r,\lambda\right) }{Z_{q}\left( b,\Phi _{\lambda +q}\right) }.
\end{equation}
Plugging \eqref{num} and \eqref{den} in \eqref{lapx=0}, we get
\begin{eqnarray}\label{x==0}
\mathbb{E}\left[ \mathrm{e}^{-q\parisiantime_{r}^{\lambda }}\ind_{\left\{\parisiantime_{r}^{\lambda }<\tau _{b}^{+}\right\} }\right]  &=&\frac{\lambda }{\lambda +q}\left( 1-\frac{Z_{q}\left( b\right) +\Lambda^{(q)}\left( b;r\right) -\Lambda^{(q)}\left( b;r,-q \right) }{Z_{q}\left(
b,\Phi _{\lambda +q}\right) \mathrm{e}^{\left( \lambda +q\right) r}+\Lambda
^{(q)}\left( b;r\right) -\Lambda ^{(q)}\left( b;r,\lambda \right) }\right)\notag \\
&=&\frac{\lambda }{\lambda +q}\left( 1-\frac{S^{(q,\lambda)}\left( b,r\right) }{\Theta^{\left( q \right) }\left( b;r,\lambda\right) }\right). 
\end{eqnarray}
We finally obtain the result by plugging the last expectation in \eqref{allx} together with \eqref{Lem1} and \eqref{Lem2}. \\
To prove \eqref{th2}, we use again the strong Markov property and spectral negativity of $X$, we have for $x<0$ 
\begin{equation}\label{xpos}
\mathbb{E}_{x}\left[ \mathrm{e}^{-q\tau _{b}^{+}}\ind_{\left\{ \tau
_{b}^{+}<\parisiantime_{r}^{\lambda }\right\} }\right] =\mathbb{E}\left[ \mathrm{e}^{-q\tau _{b}^{+}}\ind_{\left\{ \tau_{b}^{+}<\parisiantime_{r}^{\lambda }\right\} }\right] \mathbb{E}_{x}\left[ \mathrm{e}^{-q\tau _{0}^{+}}\ind_{\left\{ \tau
_{0}^{+}<r\right\} }\right].
\end{equation}
For $0\leq x\leq b$ and using \eqref{ParLap1} and \eqref{xpos}, we get 
\begin{eqnarray*}
\mathbb{E}_{x}\left[ \mathrm{e}^{-q\tau _{b}^{+}}\ind_{\left\{ \tau
_{b}^{+}<\parisiantime_{r}^{\lambda }\right\} }\right]  &=&\mathbb{E}_{x}\left[ 
\mathrm{e}^{-q\tau _{b}^{+}}\ind_{\left\{ \tau _{b}^{+}<T_{0}^{-}\right\} }%
\right] +\mathbb{E}_{x}\left[ \mathrm{e}^{-qT_{0}^{-}}\mathbb{E}_{X_{T^{-}_{0}}}\left[ 
\mathrm{e}^{-q\tau _{b}^{+}}\ind_{\left\{ \tau _{b}^{+}<\parisiantime_{r}^{\lambda
}\right\} }\right]\ind_{\left\{ T_{0}^{-}<\tau _{b}^{+}\right\} }\right]  \\
&=&\mathbb{E}_{x}\left[ \mathrm{e}^{-q\tau _{b}^{+}}\ind_{\left\{ \tau
_{b}^{+}<T_{0}^{-}\right\} }\right]  \\
&&+\mathbb{E}\left[ \mathrm{e}^{-q\tau _{b}^{+}}\ind_{\left\{ \tau
_{b}^{+}<\parisiantime_{r}^{\lambda }\right\} }\right] \mathbb{E}_{x}\left[ \mathrm{e}%
^{-qT_{0}^{-}}\mathbb{E}_{_{X_{T_{0}^{-}}}}\left[ \mathrm{e}^{-q\tau
_{0}^{+}}\ind_{\left\{ \tau _{0}^{+}<r\right\} }\right]\ind_{\left\{
T_{0}^{-}<\tau _{b}^{+}\right\} }\right]  \\
&=&\frac{Z_{q}\left( x,\Phi _{\lambda +q}\right) }{Z_{q}\left( b,\Phi
_{\lambda +q}\right) }\\
&&+\me^{-qr}\mathbb{E}\left[ \mathrm{e}^{-q\tau
_{b}^{+}}\ind_{\left\{ \tau _{b}^{+}<\parisiantime_{r}^{\lambda }\right\} }\right] 
\mathbb{E}_{x}\left[\mathrm{e}^{-qT_{0}^{-}}\Lambda^{(q)}(X_{T_{0}^{-}},r) \ind_{\left\{ T_{0}^{-}<\tau _{b}^{+}\right\}}\right]. 
\end{eqnarray*}
Setting $x=0$, yields
\begin{equation*}
\mathbb{E}\left[ \mathrm{e}^{-q\tau _{b}^{+}}\mathbf{1}_{\left\{ \tau
_{b}^{+}<\parisiantime_{r}^{\lambda }\right\} }\right] =\frac{\frac{1}{%
Z_{q}\left( b,\Phi _{\lambda +q}\right) }}{\mathrm{e}^{\lambda r}+\mathrm{e}%
^{-qr}\frac{\Lambda ^{(q)}\left( b;r\right) -\Lambda ^{(q)}\left(
b;r,\lambda \right) }{Z_{q}\left( b,\Phi _{\lambda +q}\right) }}=\frac{%
\mathrm{e}^{qr}}{\Theta ^{\left( q \right) }\left( b;r,\lambda\right) }.
\end{equation*}
Then, putting all the pieces together, we have 
\begin{eqnarray*}
\mathbb{E}_{x}\left[ \mathrm{e}^{-q\tau _{b}^{+}}\mathbf{1}_{\left\{ \tau
_{b}^{+}<\parisiantime_{r}^{\lambda }\right\} }\right]  &=&\frac{Z_{q}\left(
x,\Phi _{\lambda +q}\right) }{Z_{q}\left( b,\Phi _{\lambda +q}\right) }+%
\frac{\mathbb{E}_{x}\left[\mathrm{e}^{-qT_{0}^{-}}\Lambda^{(q)}(X_{T_{0}^{-}},r) \ind_{\left\{ T_{0}^{-}<\tau _{b}^{+}\right\}}\right] }{\Theta ^{\left( q\right) }\left(b;r,\lambda\right) } \\
&=&\frac{\Theta ^{\left( q \right) }\left( x;r,\lambda\right) }{\Theta^{\left( q\right) }\left( b;r,\lambda\right) }.
\end{eqnarray*}%
To deal with the limit as $b\rightarrow \infty $, we apply the same machinery using identity \eqref{LL2}. We can also compute the limit directly using \eqref{CM1} and \eqref{limitZW} and the fact that 
\begin{eqnarray*}
\lim_{b\rightarrow \infty }\frac{\mathcal{W}_{z}^{\left( q,\lambda \right)
}\left( b+z\right) }{W_{q}\left( b\right) } &=&\me^{\Phi_{q}z}+\lambda \lim_{b\rightarrow \infty }\int_{0}^{z}\frac{W_{q}\left( b+z-y\right) }{W_{q}\left( b\right) }%
W_{ q+\lambda}\left( y\right) \md y\qquad  \\
&=&\me^{\Phi_{q}z}\left( 1+\lambda \int_{0}^{z}\me^{-\Phi_{q}y}W_{q+\lambda }\left( y\right) \md y\right) =Z_{q+\lambda }\left( z,\Phi
_{q}\right) .
\end{eqnarray*}
Then,
\begin{eqnarray*}
\lim_{b\rightarrow \infty }\frac{\mathcal{S}^{\left( q \right)
}\left( b,r\right) }{\Theta ^{\left( q \right) }\left( b;r,\lambda\right) }
&=&\lim_{b\rightarrow \infty }\frac{\mathcal{S}^{\left( q \right)
}\left( b,r\right) /W_{q}\left( b\right) }{\Theta ^{\left( q \right)
}\left( b;r,\lambda\right) /W_{q}\left( b\right) } \\
&=&\frac{q/\Phi _{q}-\int_{0}^{\infty }\left( Z\left( z,\Phi _{q}\right) -%
\mathrm{e}^{\Phi _{q}z}\right) \frac{z}{r}\mathbb{P}\left( X_{r}\in \mathrm{d%
}z\right) }{\mathrm{e}^{\left( \lambda +q\right) r}\lambda /\left( \Phi
_{\lambda +q}-\Phi _{q}\right) -\int_{0}^{\infty }\left( Z_{q+\lambda
}\left( z,\Phi _{q}\right) -\mathrm{e}^{\Phi _{q}z}\right) \frac{z}{r}%
\mathbb{P}\left( X_{r}\in \mathrm{d}z\right) }.
\end{eqnarray*}

\section{Example}
We now present an example for which we can compute easily the probability of Parisian ruin as given in Equation~\eqref{ruinP1v2}. Note that, to use the formula in~\eqref{ruinP1v2}, one needs to have an expression for the scale function and the distribution of the underlying Lévy risk process $X$.

\subsection{Brownian risk process}

Let $X$ be a Brownian risk processes, i.e.\ 
\begin{equation*}
X_t - X_0 = c t + B_t ,
\end{equation*}
where $B=\{B_t, t\geq 0\}$ is a standard Brownian motion.

In this case, for $x\geq 0$ and $q>0$, the scale functions are given by 
\begin{align*}
\Phi _{\lambda }& = \sqrt{c^{2}+2\lambda}-c, \\
W(x)& =\frac{1}{c}\left( 1-\mathrm{e}^{-2cx}\right) ,%
\text{ \ }W_\lambda(x)=\frac{1}{\Phi _{\lambda }+c}\left( \mathrm{e}^{\Phi _{\lambda }x}-\mathrm{e}^{-x\left( \Phi _{\lambda}+2c\right) }\right) , \\
Z_{\lambda }(x)& =\frac{\lambda }{\Phi _{\lambda }+c}\left( \frac{%
\mathrm{e}^{\Phi _{\lambda }x}}{\Phi _{\lambda }}+\frac{\mathrm{e}^{-\left(
\Phi _{\lambda }+2c\right) x}}{\Phi _{\lambda }+2c}%
\right) , \\
Z\left( x,\Phi _{\lambda }\right) & =\frac{\lambda }{c}\left( \frac{1}{\Phi
_{\lambda }}-\frac{\mathrm{e}^{-2cx}}{\Phi _{\lambda }+2c}\right),
\end{align*}
Using \eqref{eq:convsnlp}, we have
\begin{eqnarray*}
\mathcal{W}_{z}^{\left( \lambda ,-\lambda \right) }\left( x+z\right)  &=& W(x+z)+\lambda \int_{0}^{z}W\left( x+z-y\right) W_{\lambda }\left( y\right) 
\mathrm{d}y\\
&=&\mathrm{e}^{\Phi _{\lambda }z}\left( \frac{\lambda }{c\Phi _{\lambda }\left( \Phi
_{\lambda }+c\right) }-\frac{\lambda \mathrm{e}^{-2cx}}{c\left( \Phi
_{\lambda }+2c\right)\left(\Phi _{\lambda }+c\right) }\right) 
\\
&&+\mathrm{e}^{-\left( \Phi _{\lambda }+2c\right) z}\left( \frac{\lambda }{c\left( \Phi _{\lambda }+c\right)
\left( \Phi _{\lambda }+2c\right) }-\frac{\lambda \mathrm{e}%
^{-2cx}}{c \Phi _{\lambda }\left( \Phi _{\lambda }+c\right) }\right)  \\
&=&\mathrm{e}^{\Phi _{\lambda }z}A_{1}(x)+\mathrm{e}^{-\left( \Phi _{\lambda
}+2c\right) z}A_{2}(x),
\end{eqnarray*}
where
\begin{eqnarray*}
A_{1}(x) &=&\frac{\lambda }{c\Phi _{\lambda }\left( \Phi
_{\lambda }+c\right) }-\frac{\lambda \mathrm{e}^{-2cx}}{c\left( \Phi
_{\lambda }+2c\right)\left(\Phi _{\lambda }+c\right) }, \\
A_{2}(x) &=&\frac{\lambda }{c\left( \Phi _{\lambda }+c\right)
\left( \Phi _{\lambda }+2c\right) }-\frac{\lambda \mathrm{e}%
^{-2cx}}{c \Phi _{\lambda }\left( \Phi _{\lambda }+c\right) }.
\end{eqnarray*}
%
Making the change of variable $y=\left( z-r\sqrt{c^{2}+2\lambda }\right) /%
\sqrt{r}$, we have 
\begin{eqnarray}
&&\frac{1}{\sqrt{2\pi r}}\int_{0}^{\infty }\mathrm{e}^{\Phi _{\lambda }z}%
\frac{z}{r}\mathrm{e}^{-\left( z-cr\right) ^{2}/\left( 2r\right) }\mathrm{d}z
\notag  \label{1} \\
&=&\frac{1 }{\sqrt{2r\pi }}\mathrm{e}^{-rc^{2}/ 2 }+e\,^{r\lambda }\left(\Phi _{\lambda}+c\right)\mathcal{N}\left(\sqrt{r}\left(\Phi _{\lambda
}+c\right) \right)   \notag \\
&=&\Psi _{1}\left( c,r,\lambda \right) ,
\end{eqnarray}%
and, setting $y=-\left( z+r\sqrt{c^{2}+2\lambda }\right) /\sqrt{r}$, we get 
\begin{eqnarray}
&&\frac{1}{\sqrt{2\pi r}}\int_{0}^{\infty }\mathrm{e}^{-\left( \Phi
_{\lambda }+2c\right) z}\frac{z}{r}\mathrm{e}^{-\left(
z-cr\right) ^{2}/\left( 2r\right) }\mathrm{d}z  \notag  \label{2} \\
&=&\frac{1 }{\sqrt{2r\pi }}\mathrm{e}^{-rc^{2}/2 }-\mathrm{e}^{r\lambda }\left(\Phi _{\lambda}-c\right) \mathcal{N}\left( -\sqrt{r}\left( \Phi _{\lambda
}-c\right) \right)   \notag \\
&=&\Psi _{2}\left( c,r,\lambda \right) ,
\end{eqnarray}%
where $\mathcal{N}$ is the cumulative distribution of the standard normal
distribution. Using~\eqref{1} and~\eqref{2}, we obtain 
\begin{equation*}
\mathbb{E}\left[ X_{1}\right] \int_{0}^{\infty }\mathcal{W}_{z}^{\left(
\lambda ,-\lambda \right) }\left( x+z\right) \frac{z}{r}\mathbb{P}\left(
X_{r}\in \mathrm{d}z\right) =A_{1}(x)\Psi _{1}\left( c,r,\lambda \right)
+A_{2}(x)\Psi _{2}\left( c,r,\lambda \right) .
\end{equation*}%
Also, we have 
\begin{equation*}
\mathbb{E}\left[ X_{1}\right] \Lambda \left( x,r\right) =\left( 1-\mathrm{e}%
^{-2xc}\right) \frac{ \mathrm{e}^{-c^{2}r/2 }}{\sqrt{2r\pi }}+c\mathcal{N}\left(\sqrt{r}c\right) \left( 1+\mathrm{e}^{-2xc}\right) ,
\end{equation*}%
while the denominator of \eqref{ruinP1} is given by 
\begin{eqnarray*}
\int_{0}^{\infty }\left( Z_{q}(z)-1\right) \frac{z}{r}\mathbb{P}\left(
X_{r}\in \mathrm{d}z\right)&=&\frac{\lambda }{\Phi _{\lambda }(\Phi _{\lambda }+c)}\Psi
_{1}\left( c,r,\lambda \right) +\frac{\lambda }{\left( \Phi
_{\lambda }+c\right) \left( \Phi _{\lambda }+2c\right) }\Psi
_{2}\left( c,r,\lambda \right)  \\
&&-\left( \frac{1}{\sqrt{2r\pi }}\mathrm{e}^{-c^{2}r/2 }+c\mathcal{N}\left(c\sqrt{r}\right) \right) .
\end{eqnarray*}%

Putting all the pieces together, we get an expression for the probability of Parisian ruin in \eqref{ruinP1}.
\section{Acknowledgements}
Funding in support of this work was provided by the Institut des sciences math\'ematiques (ISM) and the Faculté des sciences at UQAM (PhD scholarships). The author thanks two anonymous referees for their very helpful comments and suggestions.
\bibliographystyle{alpha}
\bibliography{REFERENCES}

\begin{bibdiv}
\begin{biblist}

\bib{albrecheretal2016}{article}{
      author={Albrecher, H.},
      author={Ivanovs, J.},
      author={Zhou, X.},
       title={Exit identities for {L}\'evy processes observed at {P}oisson
  arrival times},
        date={2016},
        ISSN={1350-7265},
     journal={Bernoulli},
      volume={22},
      number={3},
       pages={1364\ndash 1382},
  url={http://dx.doi.org.proxy.bibliotheques.uqam.ca:2048/10.3150/15-BEJ695},
      review={\MR{3474819}},
}

\bib{baurdoux_et_al_2015}{article}{
      author={Baurdoux, E.~J.},
      author={Pardo, J.~C.},
      author={P\'erez, J.~L.},
      author={Renaud, J.-F.},
       title={Gerber-{S}hiu distribution at {P}arisian ruin for {L}\'evy
  insurance risk processes},
        date={2016},
     journal={J. Appl. Probab.},
}

\bib{bertoin1996}{book}{
      author={Bertoin, J.},
       title={L\'evy processes},
   publisher={Cambridge University Press},
        date={1996},
}

\bib{czarna_Renaud2015}{article}{
      author={Czarna, I.},
      author={Renaud, J.-F.},
       title={A note on parisian ruin with an ultimate bankruptcy level for
  {L}évy insurance risk processes},
        date={2016},
     journal={Statist. Probab. Lett.},
}

\bib{guerin_renaud_2015}{article}{
      author={Gu\'erin, H.},
      author={Renaud, J.-F.},
       title={On the distribution of cumulative {P}arisian ruin},
     journal={arXiv:1509.06857 [math.PR]},
}

\bib{landriaultetal2018}{article}{
      author={Landriault, D.},
      author={Li, B.},
      author={Wong, J.~T.Y.},
      author={Xu, D.},
       title={Poissonian potential measures for {L}évy risk models},
        date={2018},
        ISSN={0167-6687},
     journal={Insurance: Mathematics and Economics},
      volume={82},
       pages={152 \ndash  166},
  url={http://www.sciencedirect.com/science/article/pii/S0167668718300520},
}

\bib{landriaultetal2011}{article}{
      author={Landriault, D.},
      author={Renaud, J.-F.},
      author={Zhou, X.},
       title={Occupation times of spectrally negative {L}\'evy processes with
  applications},
        date={2011},
     journal={Stochastic Process. Appl.},
      volume={121},
      number={11},
       pages={2629\ndash 2641},
}

\bib{landriaultetal2014}{article}{
      author={Landriault, D.},
      author={Renaud, J.-F.},
      author={Zhou, X.},
       title={An insurance risk model with {P}arisian implementation delays},
        date={2014},
     journal={Methodol. Comput. Appl. Probab.},
      volume={16},
      number={3},
       pages={583\ndash 607},
}

\bib{binetal2016}{article}{
      author={Li, B.},
      author={Willmot, G.~E.},
      author={Wong, J. T.~Y.},
       title={A temporal approach to the {P}arisian risk model},
        date={2018},
     journal={J. Appl. Probab.},
      volume={55},
      number={1},
       pages={302\ndash 317},
}

\bib{MR3257360}{article}{
      author={Li, Y.},
      author={Zhou, X.},
       title={On pre-exit joint occupation times for spectrally negative
  {L}\'evy processes},
        date={2014},
        ISSN={0167-7152},
     journal={Statist. Probab. Lett.},
      volume={94},
       pages={48\ndash 55},
  url={https://doi-org.proxy.bibliotheques.uqam.ca:2443/10.1016/j.spl.2014.06.023},
      review={\MR{3257360}},
}

\bib{lkabousetal2016}{article}{
      author={Lkabous, M.~A.},
      author={Czarna, I.},
      author={Renaud, J.-F.},
       title={Parisian ruin for a refracted {L}\'evy process},
        date={2017},
        ISSN={0167-6687},
     journal={Insurance Math. Econom.},
      volume={74},
       pages={153\ndash 163},
         url={http://dx.doi.org/10.1016/j.insmatheco.2017.03.005},
      review={\MR{3648884}},
}

\bib{lkabous-renaud2018risks}{article}{
      author={Lkabous, M.~A.},
      author={Renaud, J.-F.},
       title={A {V}a{R}-type risk measure derived from cumulative {P}arisian
  ruin for the classical risk model},
        date={2018},
     journal={Risks},
      volume={6},
      number={3},
}

\bib{lkabous_renaud2017}{article}{
      author={Lkabous, M.~A.},
      author={Renaud, J.-F.},
       title={A unified approach to ruin probabilities with delay for
  spectrally negative {L}\'evy processes},
        date={submitted},
}

\bib{loeffenetal2013}{article}{
      author={Loeffen, R.~L.},
      author={Czarna, I.},
      author={Palmowski, Z.},
       title={Parisian ruin probability for spectrally negative {L}\'evy
  processes},
        date={2013},
     journal={Bernoulli},
      volume={19},
      number={2},
       pages={599\ndash 609},
}

\bib{loeffenetal2017}{article}{
      author={Loeffen, R.~L.},
      author={Palmowski, Z.},
      author={Surya, B.~A.},
       title={Discounted penalty function at {P}arisian ruin for {L}\'evy
  insurance risk process},
        date={2017},
     journal={Insurance Math. Econom.},
}

\bib{loeffenetal2014}{article}{
      author={Loeffen, R.~L.},
      author={Renaud, J.-F.},
      author={Zhou, X.},
       title={Occupation times of intervals until first passage times for
  spectrally negative {L}\'evy processes},
        date={2014},
     journal={Stochastic Process. Appl.},
      volume={124},
      number={3},
       pages={1408\ndash 1435},
}

\end{biblist}
\end{bibdiv}

\end{document}